\documentclass[12pt,reqno]{amsart}

\usepackage{amssymb}
\usepackage[pdftex]{hyperref}
\hypersetup{colorlinks=true, linkcolor=black, citecolor=black}

\newcommand{\C}{{\mathbb C}}

\newcommand{\R}{{\mathbb R}}
\newcommand{\Z}{{\mathbb Z}}

\newcommand{\E}{{\mathsf E}}

\renewcommand{\o}[1]{\overline{#1}}

\newcommand{\oA}{{\o{A}}}
\newcommand{\oD}{{\o{D}}}

\newcommand{\del}{\delta}
\newcommand{\eps}{\varepsilon}
\newcommand{\lam}{\lambda}
\newcommand{\sig}{\sigma}
\newcommand{\ome}{\omega}

\newcommand{\Lam}{\Lambda}
\newcommand{\Ome}{\Omega}

\newcommand{\lfl}{\left\lfloor}
\newcommand{\rfl}{\right\rfloor}

\newcommand{\lcl}{\left\lceil}
\newcommand{\rcl}{\right\rceil}

\newcommand{\longc}{,\dotsc,}

\newcommand{\longcap}{\cap\dotsb\cap}

\newcommand{\est}{\varnothing}
\newcommand{\seq}{\subseteq}
\newcommand{\stm}{\setminus}

\newtheorem{claim}{Claim}
\newtheorem{lemma}{Lemma}
\newtheorem{theorem}{Theorem}
\newtheorem{corollary}{Corollary}
\newtheorem{proposition}{Proposition}

\DeclareMathOperator{\supp}{supp}

\newcommand{\refc}[1]{\ref{c:#1}}
\newcommand{\refl}[1]{\ref{l:#1}}
\newcommand{\refm}[1]{\ref{m:#1}}
\newcommand{\reft}[1]{\ref{t:#1}}
\newcommand{\refp}[1]{\ref{p:#1}}
\newcommand{\refs}[1]{\ref{s:#1}}
\newcommand{\refb}[1]{\cite{b:#1}}
\newcommand{\refe}[1]{\eqref{e:#1}}

\newcommand{\sub}[1]{_{\substack{#1}}}

\title{The popularity gap}

\author{Vsevolod F. Lev}
\address[Vsevolod Lev]{Department of mathematics,
  the University of Haifa at Oranim, Tivon 36006, Israel}
\email{seva@math.haifa.ac.il}

\author{Ilya D. Shkredov}
\address[Ilya Shkredov]{Steklov Mathematical Institute, ul. Gubkina, 8,
    Moscow, Russia, 119991}
\email{ilya.shkredov@gmail.com}

\begin{document}
\baselineskip=16pt

\begin{abstract}
Suppose that $A$ is a finite, nonempty subset of a cyclic group of either
infinite or prime order. We show that if the difference set $A-A$ is ``not
too large'', then there is a nonzero group element with at least as many as
$(2+o(1))|A|^2/|A-A|$ representations as a difference of two elements of
$A$; that is, the second largest number of representations is, essentially,
twice the average. Here the coefficient $2$ is best possible.

We also prove continuous and multidimensional versions of this result, and
obtain similar results for sufficiently dense subsets of an arbitrary
abelian group.
\end{abstract}

\maketitle

\section{Background and summary of results}\label{s:intro}

A large body of problems, conjectures, and results in additive combinatorics
assert, in different ways, that, normally, the number-of-representations
function exhibits an irregular behaviour. The general framework for this line
of research is as follows.

For finite subsets $A$ and $B$ and an element $g$ of an abelian group, let
$r_{A,B}(g):=|A\cap(g-B)|$; that is, $r_{A,B}(g)$ is the number of
representations of $g$ as a sum of an element of $A$ and an element of $B$.
How close to a constant function can the function $r_{A,B}$ be?

In this paper we consider representations by differences; that is, we are
concerned with the special case where $B=-A$. Notation-wise, we abbreviate
$r_{A,-A}$ as $r_A$.

Clearly, the largest value attained by the function $r_A$ is $|A|$. Our
principal results show that for the cyclic groups of either infinite or prime
order, the \emph{second} largest value attained by $r_A$ is, essentially, at
least twice the average value.

For a prime $p$, by $\C_p$ we denote the cyclic group of order $p$.
\begin{theorem}\label{t:modp}
Suppose that $p$ is a prime, $A\seq\C_p$ is a nonempty subset, and
$\del\in(0,1/3)$. If
  $$ |A-A| \le \min\{K|A|, 2(p+1)/3 \},\quad K<|A|^{\del}, $$
then there is a nonzero element $d\in\C_p$ such that
$r_A(d)>2K^{-1}|A|(1-2\del\ln(2/\del))$.
\end{theorem}

The following integer analog of Theorem~\reft{modp} is, in fact, its
immediate consequence.
\begin{theorem}\label{t:intverD}
Suppose that $A$ is a finite, nonempty set of integers, and that
$\del\in(0,1/3)$. If
  $$ |A-A| \le K|A|,\quad K<|A|^{\del}, $$
then there is an integer $d\ne 0$ such that
$r_A(d)>2K^{-1}|A|(1-2\del\ln(2/\del))$.
\end{theorem}

Comparing Theorems~\reft{modp} and~\reft{intverD}, we see that the bounds
obtained are not affected by the ``modulo-$p$ overlapping''. Loosely
speaking, there is no difference between the group $\C_p$ and the group of
integers $\Z$ as far as the second largest value of the function $r_A$ is
concerned.

The following theorem establishes a similar estimate in terms of the
\emph{length} of the set $A$ instead of the size of its difference set.
\begin{theorem}\label{t:intverL} Suppose that $L>1$, and that $A\seq[1,L]$ is a
nonempty set of integers. If
 $|A-A|<|A|^{1+\del}$ with $\del\in(0,1/3)$, then there is an integer
$d\ne 0$ such that $r_A(d)>\frac{|A|^2}{L}\,(1-2\del\ln(2/\del))$.
\end{theorem}

We note that the assumption $|A-A|<|A|^{1+\del}$ of Theorem~\reft{intverL}
holds trivially if $A$ is sufficiently dense in $[1,L]$; say,
$|A|>(2L)^{1/(1+\del)}$ suffices.

Theorem~\reft{intverL} follows readily from Theorem~\reft{intverD} by letting
$K:=|A-A|/|A|$ and observing that then $2K^{-1}|A|=2|A|^2/|A-A|>|A|^2/L$.

Apart from the factor $1-2\del\ln(2/\del)$, the lower bounds obtained in
Theorems~\reft{modp}--\reft{intverL} are, essentially, best possible; they
are attained, for example, when $A$ is an interval or an appropriately
sampled random set.

Interestingly, Theorems~\reft{modp}--\reft{intverL} cannot be obtained by a
straightforward averaging as the number of elements $d\in A-A$ with $r_A(d)$
large can have ``zero measure''.
\begin{claim}\label{m:measure0}
For any $\eps\in(0,1)$ there is a finite, nonempty set $A\subset\Z$ such that
\begin{equation*}\label{e:measure0}
  \Big| \Big\{ d \colon
     r_A(d) \ge (1-\eps)\frac{2|A|^2}{|A-A|} \Big\} \Big| \le 2\eps |A-A|.
\end{equation*}
\end{claim}

Since the function $r_A\colon d\mapsto|A\cap(A+d)|$ is generalized by the
functions $(d_1\longc d_{k-1})\mapsto|A\cap(A+d_1)\longcap(A+d_{k-1})|$ (see
the next section), it is natural to extend Theorem~\reft{modp} onto the
intersection of $k\ge 3$ translates of the set $A$. In this direction, we
prove the following multidimensional generalization.
\begin{theorem}\label{t:extD}
Suppose that $p$ is a prime, $A\seq\C_p$ is a nonempty subset, and
$\del\in(0,1/(3k))$ with an integer $k\ge 3$. If
  $$ |A-A| \le \min\{K|A|,2(p+1)/3 \},\quad K<|A|^{\del}, $$
then there are pairwise distinct, nonzero elements $d_1\longc d_{k-1}\in\C_p$
such that
  $$ |A\cap(A+d_1)\longcap(A+d_{k-1})|
                           > 2^{k-1}K^{-(k-1)}|A|(1-3\del k^2\ln(1/k\del)). $$
\end{theorem}

From Theorem~\reft{extD} one can easily derive multidimensional analogs of
Theorems~\reft{intverD} and~\reft{intverL}; we omit the details.

Our next result exhibits irregularities in the behaviour of the function
$r_A$ for large subsets $A$ of an \emph{arbitrary} finite abelian group.
\begin{theorem}\label{t:arbG}
Suppose that $A$ is a subset of a finite abelian group $G$ such that $|A|\ge
2^{10}$ and $|A-A|=(1-\eps)|G|$ with $0<\eps\le2^{-5}$. If
$|A-A|\le|A|^{1+\sqrt{\eps}/2}$, then there is a nonzero element $d\in G$
such that $r_A(d) \ge
\frac{|A|^2}{|A-A|}\,\left(1+\frac18\,\sqrt\eps\right)$.
\end{theorem}

Finally, we prove a continuous analog of Theorem~\reft{intverL}. For a
function $f\in L^1(\R)$, let
  $$ (f\circ f)(x):=\int_{\R} f(t)f(x+t)\,dt. $$
\begin{theorem}\label{t:fan}
Suppose that $f$ is a real, nonnegative function with $\supp(f)\seq[0,1]$. If
$f$ is not constant on its support then, letting $\rho:=\|f\|_2/\|f\|_1$, for
any $0<\del\le\min\{\frac1{2\rho^{12}},\frac1{2^8}\}$ we have
\begin{equation*}\label{f:f_sup}
  \sup_{x \in [-1,1] \stm [-\del,\del]}\, \frac{(f\circ f)(x)}{\|f\|_1^2}
    \ge 1 - 8\max\Big\{ \sqrt[8]{2\del},\,
                     \frac{\ln\log_\rho(1/2\del)}{\log_\rho(1/2\del)} \Big\}.
\end{equation*}
\end{theorem}

We note that the matching upper bound $\|f\circ f\|_\infty\le\|f\|_1^2$ is
trivial.

The quantity $\log_\rho(1/2\del)$ characterizes the length of the forbidden
interval $[-\del,\del]$ as measured against the ``scatter'' of $f$. If
$\log_\rho(1/2\del)$ is small enough, then we can have
$\supp(f)\seq[0,\del]$; in this case $\supp(f\circ f)\seq[-\del,\del]$ and
there do not exist $x\notin[-\del,\del]$ with $(f\circ f) (x)>0$.
Theorem~\reft{fan} assumes $\log_\rho(1/2\del)\ge 12$, and the remainder term
is $o(1)$ in the regime where $\log_\rho(1/2\del)\to\infty$ and $\del\to 0$.

Our argument involves two major components: the higher energies
technique~\refb{ss} and a result in the spirit of~\refb{l} establishing an
extremal property of the interval subsets of prime-order groups. The central
role is played by the following quantity.

For finite subsets $A$ and $D$ of an abelian group and integer $k\ge 1$, let
$T_D^{(k)}(A)$ be the number of $k$-tuples $(a_1\longc a_k)\in A^k$ such that
$a_i-a_j\in D$ for any $1\le i,j\le k$. In particular, $T_D^{(k)}(D)$ is the
number of $k$-tuples $(d_1\longc d_k)$ such that
  $$ \{0,d_1\longc d_k\}-\{0,d_1\longc d_k\}\seq D. $$
The key result we prove about this quantity shows that in a group of prime
order, if $0\in D=(-D)$, then $T_D^{(k)}(A)$ can only get larger if $A$ and
$D$ are replaced with the intervals $\oA,\oD$ of sizes $|\oA|=|A|$ and
$|\oD|=|D|$, respectively, centered at $0$; that is, either $\oA=[-m,m]$, or
$\oA=[-(m-1),m]$ with $m=\lfloor|A|/2\rfloor$ in both cases, and similarly
for $\oD$.
\begin{proposition}\label{p:intopt}
For any prime $p\ge 5$, integer $k\ge 1$, and nonempty subsets $A,D\seq\C_p$
with $0\in D=(-D)$, we have
  $$ T_D^{(k)}(A) \le T_\oD^{(k)}(\oA), $$
where $\oA,\oD\seq\C_p$ are intervals with $|\oA|=|A|$ and $|\oD|=|D|$,
centered at $0$.
\end{proposition}

As a consequence of Proposition~\refp{intopt} we prove the following,
slightly technical, corollary.
\begin{corollary}\label{c:intopt}
For any prime $p\ge 5$, integer $k\ge 3$, and subset $D\seq\C_p$ with $0\in
D=(-D)$ and $k\le |D|\le\frac23(p+1)$, we have
  $$ T_D^{(k)}(D) \le 3k2^{-k-1}|D|^k. $$
\end{corollary}

We now turn to the proofs of the results discussed above. In the next section
we introduce the basic notation, collect auxiliary tools needed for the
proofs, and prove Claim~\refm{measure0}. In Section~\refs{intopt} we prove
the key Proposition~\refp{intopt}. Theorem~\reft{modp} is derived from
Proposition~\refp{intopt} in Section~\refs{modp} where also
Corollary~\refc{intopt} is proved; as we have noticed,
Theorems~\reft{intverD} and~\reft{intverL} follow from Theorem~\reft{modp}
and therefore do not require any additional consideration. Theorems~
\reft{arbG} and \reft{extD} are proved in Sections~\refs{arbG} and
\refs{extD}, respectively. A proof of Theorem~\reft{fan}, along with a
somewhat broader context for the problems studied in this paper, is outlined
in the concluding Section~\refs{fan}.

\section{Notation and preliminaries}\label{s:prelim}

In this section $A$ and $D$ are finite, nonempty subsets of an abelian group
$G$.

Following~\refb{ss} (but using a different notation), for an integer $k\ge 2$
we let
\begin{equation}\label{e:Rdef}
  R^{(k)}_A(x_1\longc x_{k-1}) := |A\cap(A+x_1)\longcap(A+x_{k-1})|,
                                          \quad x_1\longc x_{k-1}\in G.
\end{equation}
This quantity generalizes the function $r_A$ which is obtained as a
particular case $k=2$. Alternatively, $R^{(k)}_A(x_1\longc x_{k-1})$ is the
number of representations of $(x_1\longc x_{k-1})$ as a difference of an
element of the ``diagonal'' $\{(a\longc a)\colon a\in A\}$ and an element of
the Cartesian power $A^{k-1}$. Clearly,
\begin{equation}\label{e:sumC}
  \sum_{x_1\longc x_{k-1}\in G} R^{(k)}_A(x_1\longc x_{k-1}) = |A|^k.
\end{equation}
Less obvious is the identity
\begin{equation}\label{e:Ckl}
  \sum_{x_1\longc x_{k-1}\in G}\big(R^{(k)}_A(x_1\longc x_{k-1})\big)^l
        = \sum_{y_1\longc y_{l-1}\in G}
                              \big(R^{(l)}_A(y_1\longc y_{l-1})\big)^k
\end{equation}
valid for any $k,l\ge 2$ and any finite subset $A\seq G$,
see~\cite[Lemma~2.8]{b:sv}.

The common value of the two sums in \refe{Ckl} is denoted $\E_{k,l}(A)$;
thus, $\E_{k,l}(A)=\E_{l,k}(A)$. We write
  $$ \E_k(A) := \E_{k,2}(A)=\sum_{x\in G}|A\cap(A+x)|^k. $$

Let $\mu(A):=\max\{r_A(d)\colon d\in A-A,\ d\ne 0\}$; that is, $\mu(A)$ is
the second largest value attained by the function $r_A$. We have
\begin{equation}\label{e:Emu}
  \E_k(A) = \sum_{x\in G}|A\cap(A+x)|^k \le |A|^k + (\mu(A))^{k-1} |A|^2.
\end{equation}

Recall that in Section~\refs{intro} we have defined
  $$ T_D^{(k)}(A) := |\{(a_1\longc a_k)\in A^k
                  \colon a_i-a_j\in D,\ 1\le i,j\le k \}|,\quad k\ge 1. $$
Thus, for instance, if $0\in D$, then $T_D^{(1)}(A)=|A|$, and if $A-A\seq D$,
then $T_D^{(k)}(A)=|A|^k$. Furthermore, $T_D^{(k)}(A+g)=T_D^{(k)}(A)$ for any
group element $g$, and $T_D^{(2)}(A)$ is the total number of representations
of the elements $d\in D$ as $d=a_1-a_2$ with $a_1,a_2\in A$. The relevance of
this quantity in our context is explained by the observation that if $D=A-A$
and $(a_1\longc a_{k-1})\in\supp(R_A^{(k)})$, then $a_i\in D$ and $a_i-a_j\in
D$ for all $1\le i,j\le k-1$; as a result,
\begin{equation}\label{e:nfc}
  |\supp(R_A^{(k)})| \le T_D^{(k-1)}(D).
\end{equation}
Combining identities~\refe{sumC} and~\refe{Ckl} and estimates \refe{nfc} and
\refe{Emu}, and using the Cauchy-Schwarz inequality, we get
\begin{align}
  |A|^{2k+2}
    &=  \big( \sum_{x_1\longc x_{k}\in G}
                            R^{(k+1)}_A(x_1\longc x_{k}) \big)^2 \notag \\
    &\le |\supp(R^{(k+1)}_A)| \cdot \sum_{x_1\longc x_{k}\in G}
                       \big(R^{(k+1)}_A(x_1\longc x_{k})\big)^2 \notag \\
    &=   |\supp(R^{(k+1)}_A)|
                        \cdot \sum_{x\in D} (R_A^{(2)}(x))^{k+1} \notag \\
    &=   |\supp(R^{(k+1)}_A)| \cdot \E_{k+1}(A)\notag \\
    &\le T_D^{(k)}(D) \cdot (|A|^{k+1} + (\mu(A))^{k} |A|^2). \label{e:basic}
\end{align}
The resulting inequality is, in fact, our main path in this paper to
estimating $\mu(A)$.

We close this section with the proof of Claim~\refm{measure0} from
Section~\refs{intro}.

\begin{proof}[Proof of Claim~\refm{measure0}]
Fix an integer $n\ge\frac1\eps+2$, write $P:=[1,n]$, and let $\Lam\seq\Z$ be
a finite Sidon set (see, for instance,~\refb{o}) satisfying
$(2P-2P)\cap(2\Lam-2\Lam)=\{0\}$ and $|\Lam|>\sqrt{2/\eps}+1$. Consider the
sumset $A:=P+\Lam$. We have $|A|=|\Lam|n$ and
$|A-A|=(2n-1)(|\Lam|^2-|\Lam|+1)$. Any element $d\in A-A$ can be represented
as $d=\lam_1-\lam_2+p_1-p_2$ with $\lam_1,\lam_2\in\Lam$ and $p_1,p_2\in P$;
moreover, if $d\notin P-P$, then $\lam_1$ and $\lam_2$ are uniquely
determined by $d$, and $r_A(d)=n-|p_1-p_2|$. Therefore if
$r_A(d)>(1-\eps)\,2|A|^2/|A-A|$, then
  $$ n-|p_1-p_2| > (1-\eps)\frac{2n^2|\Lam|^2}{(2n-1)(|\Lam|^2-|\Lam|+1)}
                                                             >(1-\eps)n $$
implying $|p_1-p_2|<\eps n$. This shows that there are at most as many as
$(|\Lam|^2-|\Lam|)(2\eps n+1)$ elements $d\notin P-P$ with
$r_A(d)>(1-\eps)\,2|A|^2/|A-A|$. Since $|P-P|=2n-1$, we get at most
\begin{align*}
  (|\Lam|^2 &- |\Lam|)(2\eps n+1) + 2n-1 \\
         &=\eps(2n-1)(|\Lam|^2-|\Lam|) + (\eps+1)(|\Lam|^2-|\Lam|)
                   + 2n-1 \\
         &= \eps(2n-1)(|\Lam|^2-|\Lam|+1) + (\eps+1)(|\Lam|^2-|\Lam|+1)
                 + (2n(1-\eps) -2)
\end{align*}
elements totally. Finally, we notice that the first summand in the right-hand
side is $\eps|A-A|$ while, in view of the assumptions $n\ge\frac1\eps+2$ and
$|\Lam|>\sqrt{2/\eps}+1$, the last two summands can be estimated as
  $$ (\eps+1)(|\Lam|^2-|\Lam|+1)
        < \frac12\,\eps\,(2n-1)(|\Lam|^2-|\Lam|+1) = \frac12\eps\,|A-A| $$
and
  $$ 2n(1-\eps) -2 < 2n-1 < \frac12\,\eps\,(2n-1)(|\Lam|^2-|\Lam|+1)
                                               = \frac12\eps\,|A-A|. $$
\end{proof}

\section{Proof of Proposition~\refp{intopt}}\label{s:intopt}

We start with a lemma establishing two simple identities that the quantities
$T_D^{(k)}(A)$ satisfy.
\begin{lemma}\label{e:id}
For any integer $k\ge 1$ and any finite subsets $A,D$ of an abelian group
with $0\in D=(-D)$ we have
\begin{align}
  T_D^{(k+1)}(A)
    &= \sum_{\sub{d_1\longc d_k\in D \\ d_i-d_j\in D}}
             R_A^{(k+1)}(d_1\longc d_k) \label{e:Ax1}
  \intertext{(with the indices $i$ and $j$ running independently over all
values in $\{1\longc k\}$), and}
    T_D^{(k+1)}(A) &= \sum_{a\in A} T_D^{(k)}(A\cap(D+a)) \label{e:Ax2}.
\end{align}
\end{lemma}

\begin{proof}
From the definition,
  $$ T_D^{(k+1)}(A) = |\{(a_1\longc a_{k+1})\in A^{k+1}
                             \colon a_i-a_j\in D,\ 1\le i,j\le k+1\}|. $$
Letting $d_i:=a_{k+1}-a_i\ (i=1\longc k)$, we rewrite this equality as
\begin{align*}
  T_D^{(k+1)}(A)
    &= \sum_{\sub{d_1\longc d_k\in D \\ d_i-d_j\in D}}
         |\{ a_{k+1}\in A\colon a_{k+1}-d_1\longc a_{k+1}-d_k\in A \}| \\
    &= \sum_{\sub{d_1\longc d_k\in D \\ d_i-d_j\in D}}
                                           |A\cap(A+d_1)\longcap(A+d_k)|,
\end{align*}
proving \refe{Ax1}. In a similar way, re-denoting $a_{k+1}$ by $a$, we get
\begin{align*}
   T_D^{(k+1)}(A)
      &= \sum_{a\in A} |\{ (a_1\longc a_k)\in A^k \colon
                   a_1-a\longc a_k-a\in D,\ a_i-a_j\in D \}| \\
      &= \sum_{a\in A} |\{ (a_1\longc a_k)\in (A\cap(D+a))^k \colon
                                           a_i-a_j\in D \}| \\
      &= \sum_{a\in A} T_D^{(k)}(A\cap(D+a))
\end{align*}
which establishes \refe{Ax2}.
\end{proof}

For $1\le n\le p$ we let $J_n:=[1,n]\seq\C_p$.

%
\begin{lemma}\label{l:conv}
For any $k\ge 2$ and any fixed $d_1\longc d_{k-1}\in\C_p$, the sequence
  $$ \big\{ R_{J_n}^{(k)}(d_1\longc d_{k-1})\colon 1\le n\le p-1 \big\} $$
is convex.
\end{lemma}

\begin{proof}
Let
  $$ B_n := J_n \cap(J_n+d_1)\longcap(J_n+d_{k-1}); $$
thus, $|B_n|=R_{J_n}^{(k)}(d_1\longc d_{k-1})$. We have $B_1\seq
B_2\seq\dotsb$ and $B_n+1\seq B_{n+1}$. Furthermore, we observe that if
$x,x+1\in B_n$ for some $x\in\C_p$ and $2\le n\le p-1$, then indeed $x\in
B_{n-1}$; we leave the verification to the reader.

To prove that the sequence $|B_n|$ is convex, we show that
  $$ (B_n\stm B_{n-1})+1 \seq B_{n+1}\stm B_n,\quad 2\le n\le p-2 $$
(which implies $|B_n|-|B_{n-1}|\le|B_{n+1}|-|B_n|$). Indeed, if $x\in B_n\stm
B_{n-1}$, then $x+1\in B_{n+1}$, while from $x\in B_n$ and $x\notin B_{n-1}$
we derive that $x+1\notin B_n$ by the observation just made.
\end{proof}

For a finite subset $S$ of a cyclic group of either infinite or prime order,
we define $\o{S}:=[-(m-1),m]$ if $|S|=2m$, and $\o{S}:=[-m,m]$ if $|S|=2m+1$
(with an integer $m\ge 0$); that is, $\o{S}$ is the interval of size
$|\o{S}|=|S|$, contained in the same underlying group, and centered around
$0$, possibly with a small ``right-end overweight''.

\begin{proof}[Proof of Proposition~\refp{intopt}]
We use induction by $k$. For $k=1$ we have the equalities
  $$ T_D^{(1)}(A) = |A| = |\oA| = T_\oD^{(1)}(\oA); $$
assume now that $k\ge 2$.

For integer $1\le n\le p-1$ let $t(n):=T_\oD^{(k)}(\o{J_n})$; equivalently,
$t(n)=T_\oD^{(k)}(J_n)$. Since the sum of convex sequences is convex, by
\refe{Ax1} and Lemma~\refl{conv}, the sequence $t(n)$ is convex in the range
$1\le n\le p-1$.

From~\refe{Ax2} and the induction hypotheses,
\begin{align*}
   T_D^{(k+1)}(A)
      &= \sum_{a\in A} T_D^{(k)}(A\cap(D+a)) \\
      &\le \sum_{a\in A} T_\oD^{(k)}(\o{A\cap(D+a)}) \\
      &= \sum_{a\in A} t(|A\cap(D+a)|).
\end{align*}
By \cite[Theorem 1]{b:l}, the sequence $\{|A\cap(D+a)|\colon a\in A\}$ is
majorized by the sequence $\{|\oA\cap(\oD+a)|\colon a\in \oA\}$, in the sense
that for any $1\le h\le|A|$, the sum of the $h$ largest terms of the former
sequence does not exceed the sum of the $h$ largest terms of the latter one.
On the other hand, as observed above, $t(n)$ is a convex sequence.
Consequently, by Karamata's inequality, we have
  $$ \sum_{a\in A} T_\oD^{(k)}(\o{A\cap(D+a)})
                       \le \sum_{a\in \oA} T_\oD^{(k)}(\o{A}\cap(\oD+a)). $$
Recalling~\refe{Ax2}, we conclude that the sum in the right-hand side is
$T_\oD^{(k+1)}(\oA)$, which proves the assertion.
\end{proof}

\section{Proof of Theorem~\reft{modp}}\label{s:modp}

\begin{lemma}\label{l:Int}
Suppose that $p\ge 5$ is a prime, and that $D=[-m,m]\seq\C_p$ with
$m=(|D|-1)/2$. If $m<\frac13\,p$, then $T_D^{(k)}(D)=(m+1)^{k+1}-m^{k+1}$.
\end{lemma}

We give two proofs of this lemma, the second being of more geometric nature.
\begin{proof}[First proof of Lemma~\refl{Int}]
Let $J:=[0,m]^k$, and for an integer $s\in[0,m]$, denote by $\vec{s}$ the
$k$-tuple $(s\longc s)$. From the assumption $m<\frac13\,p$ it is not
difficult to derive that the set of all those $k$-tuples $(x_1\longc x_k)\in
D^k$ with $x_i-x_j\in D$ for all $i,j\in[1,k]$ (which are the $k$-tuples
counted in $T_D^{(k)}(D)$) is the union of the translates $J-\vec s$ over all
$s\in[0,m]$. Therefore, using the inclusion-exclusion formula,
  $$ \textstyle T_D^{(k)}(D) = \left| \bigcup_{s=0}^m (J-\vec{s})\right|
       = \sum_{\est\ne S\seq[0,m]}
                 (-1)^{|S|-1} \left|\bigcap_{s\in S}(J-\vec{s})\right|. $$
With $s_{\mathrm{min}},\, s_{\mathrm{max}}\in[0,m]$ denoting the smallest and
the largest elements of $S$, respectively, the intersection in the right-hand
side is the cube $[-s_{\mathrm{min}},m-s_{\mathrm{max}}]^k$. Consequently,
  $$ T_D^{(k)}(D) = \sum_{l=0}^m (m+1-l)^k \sum_S (-1)^{|S|-1}, $$
where the inner sum extends onto all subsets $\est\ne S\seq[0,m]$ with
$s_{\mathrm{max}}-s_{\mathrm{min}}=l$. It remains to notice that this sum is
equal to $m+1$ if $l=0$ and to $-m$ if $l=1$, and that it vanishes for any
$l\ge 2$ (indeed, for any fixed $0\le x<y\le m$ with $y\ge x+2$ there are
equally many even-sized and odd-sized subsets $S\seq[x,y]$ with $x,y\in S$).
\end{proof}

\begin{proof}[Second proof of Lemma~\refl{Int}]
As in the first proof, let $J:=[0,m]^k$, and for $s\in[0,m]$, denote by
$\vec{s}$ the $k$-tuple $(s\longc s)$; we want to show that the translates
$J-\vec s\ (s\in[0,m])$ jointly contain as many as $(m+1)^{k+1}-m^{k+1}$
elements.

The cube $J$ has $2k$ faces, any face is obtained by setting one of the
coordinates to $0$ or $m$. Of these $2k$ faces, $k$ are ``visible'' from the
origin, with one parallel invisible face corresponding to any visible one.
Any such pair of faces contributes to $T_D^{(k)}(D)$ a parallelepiped of
volume $(m+1)^{k-1} (m+1) = (m+1)^k$ (because the distance between the two
hyperplanes containing any pair of faces is $m$). Each of these
parallelepipeds contains the set $\{-\vec{s}\colon s\in[0,m]\}$. One can
check by induction that the intersection of any $j$ parallelepipeds is a cube
of codimension $j-1$, and with the edge length $m+1$. Hence, its volume is
$(m+1)^{k+1-j}$. Also, the cube $[1,m]^k \subseteq \bigcup_{s\in [0,m]}
(J-\vec{s})$ has size $m^k$ and does not belong to any of these
parallelepipeds. Thus by the inclusion-exclusion formula we have the total
contribution of
\begin{multline*}
  \sum_{j=1}^k \binom{k}{j} (-1)^{j+1} (m+1)^{k+1-j} + m^k \\
                  = -(m+1) (m^k - (m+1)^k) + m^k = (m+1)^{k+1}-m^{k+1}.
\end{multline*}
\end{proof}

Using elementary calculus, one can show that if $|D|=2m+1\ge k\ge 3$, then
the expression $(m+1)^{k+1}-m^{k+1}$ in the statement of Lemma~\refl{Int}
does not exceed $3k2^{-k-1}|D|^k$. This establishes Corollary~\refc{intopt}
as an immediate consequence of the lemma and Proposition~\refp{intopt}.

\begin{proof}[Proof of Theorem~\reft{modp}]
We write $D:=A-A$ and let $k:=\lfloor\del^{-1}\rfloor$; thus $k\ge 2$.

The case $|A|=1$ is ruled out by the assumption $K<|A|^\del$, and we assume
that $|A|\ge 2$. By the Cauchy-Davenport theorem,
\begin{equation}\label{e:small_A}
      |A|^\del > K \ge \frac{|D|}{|A|} \ge 2-\frac1{|A|} \ge \frac 32
\end{equation}
whence $|D|\ge|A|\ge(3/2)^{1/\del}>1/\del\ge k$. Therefore,
 $T_D^{(k)}(D)\le 3k2^{-k-1}|D|^k$ by Corollary~\refc{intopt}.
Combining this estimate with \refe{basic}, we get
  $$ |A|^{2k+2} \le (|A|^{k+1}+(\mu(A))^k|A|^2)\cdot 3k2^{-k-1} |D|^k; $$
thus, at least one of
  $$ \frac58\,|A|^{2k+2} \le |A|^{k+1}\cdot 3k2^{-k-1} |D|^k $$
and
  $$ \frac38\,|A|^{2k+2} \le (\mu(A))^k|A|^2\cdot 3k2^{-k-1} |D|^k $$
holds true. This means that we have either
  $$ |A|^{k+1} \le \frac{12}5\,k2^{-k} |D|^k $$
or
  $$ |A|^{2k} \le (\mu(A))^k\cdot 4k2^{-k} |D|^k. $$

In the first case, in view of $K^k\le K^{1/\del}<|A|$, we obtain
  $$ |A| \le \frac{12}5\,k2^{-k}K^k < \frac{12}5\,k2^{-k} |A|, $$
which is impossible in view of $2^{k-2}>\frac35\,k$.

In the second case, in view of $k>1/(2\del)$, we have
  $$ \mu(A) \ge \frac{2|A|^2}{|D|} e^{-\ln(4k)/k}
     \ge 2K^{-1}|A|e^{-2\del\ln(2/\del)}
                              \ge 2K^{-1}|A|(1-2\del\ln(2/\del)). $$
\end{proof}

\section{Proof of Theorem~\reft{arbG}}\label{s:arbG}

We start with an estimate for the quantity $T_D^{(k)}(D)$ in the situation
where $D$ is a large subset of a finite abelian group.

\begin{lemma}\label{l:Sigma_large}
Suppose that $k\ge 1$ is an integer, $G$ is a finite abelian group, and
$D\seq G$ is a subset with $0\in D=(-D)$. If $|G\stm D|=\tau |D|$ with
$0<\tau\le\min\{1/2,2/(k^2-k+2)\}$, then
\begin{equation}\label{e:Sigma_large}
  T_D^{(k)}(D) \le \left( 1 - \frac14\,k(k-1) \tau\right) |D|^k.
\end{equation}
\end{lemma}

\begin{proof}
Let $C:=G\stm D$. Throughout the proof, it will be convenient to identify the
sets $C$ and $D$ with their indicator functions; thus, for instance,
$D(x)+C(x)=1$ for any $x\in G$.

Let
  $$ \Ome_k:=\{(x_1\longc x_k)\in D^k\colon x_i-x_j\in D,\ 1\le i,j\le k\}. $$

We use induction on $k$.

The case $k=1$ agrees with~\refe{Sigma_large} in view of $T^{(1)}_D(D)=|D|$.
Furthermore,
\begin{align*}
  T_D^{(2)}(D)
    &=  |\Ome_2| \\
    &=  \sum_{x\in D} |D\cap(D+x)| \\
    &=  \sum_{x\in D} (|D|-|C+x|+|C\cap(C+x)|) \\
    &\le (|D|-|C|)|D| + \sum_{x\in G} |C\cap(C+x)| \\
    &=  (|D|-|C|)|D| + |C|^2 \\
    &=  (1-\tau+\tau^2)|D|^2
\end{align*}
which, in view of the assumption $\tau\le\frac12$, settles the case $k=2$.

We now estimate the quantity $T^{(k+1)}_D(D)$ assuming that
 $k+1\ge 3$ and $\tau\le2/(k^2+k+2)$.

For any $y_1\longc y_{k}\in G$ we have
\begin{align}
   D(y_1)\dotsb D(y_{k})
     &= D(y_2)\dotsb D(y_{k}) - C(y_1)D(y_2)\dotsb D(y_{k}) \notag \\
     &= D(y_3)\dotsb D(y_{k}) - C(y_1)D(y_2)\dotsb D(y_{k})
                             - C(y_2)D(y_3)\dotsb D(y_{k})  \notag \\
     &\qquad\qquad\qquad  \vdots  \notag \\
     &= 1 - \sum_{i=1}^{k} C(y_i)D(y_{i+1})\dotsb D(y_{k}) \label{e:Dy}
\end{align}
and, similarly,
  $$ D(y_{i+1})\dotsb D(y_{k})
                    = 1-\sum_{j=i+1}^{k} C(y_j)D(y_{j+1})\dotsb D(y_k). $$
Multiplying this identity by $C(y_i)$ and substituting the result
into~\refe{Dy} we get
\begin{align*}
   D(y_1)\dotsb D(y_{k})
      &= 1 - \sum_{i=1}^{k} \Big( C(y_i)
           - \sum_{j=i+1}^k C(y_i) C(y_j) D(y_{j+1})\dotsb D(y_k) \Big) \\
      &= 1 - \sum_{i=1}^{k} C(y_i)
           + \sum_{1\le i<j\le k} C(y_i) C(y_j) D(y_{j+1})\dotsb D(y_k).
\end{align*}
We let $y_i=x-x_i$ and take sums to obtain
\begin{align*}
   T_D^{(k+1)}(D)
       &= \sum_{\sub{(x_1\longc x_{k})\in\Ome_{k} \\ x\in D}}
                                              D(x-x_1)\dotsb D(x-x_{k}) \\
       &= \sum_{\sub{(x_1\longc x_{k})\in\Ome_{k} \\ x\in D}} 1
            - \sum_{\sub{(x_1\longc x_{k})\in\Ome_{k} \\ x\in D}}
                                                \sum_{i=1}^k C(x-x_i) \\
            &\qquad\qquad +\sum_{\sub{(x_1\longc x_{k})\in\Ome_{k} \\ x\in D}}
            \sum_{1\le i<j\le k} C(x-x_i) C(x-x_j) D(x-x_{j+1})\dotsb D(x-x_k) \\
       &= \sig_1-\sig_2+\sig_3,
\end{align*}
where each of $\sig_1,\sig_2,\sig_3$ denotes the corresponding sum in the
right-hand side. Clearly, we have $\sig_1=T^{(k)}_D(D)|D|$. Furthermore,
\begin{align*}
   \sig_2
     &= k \sum_{\sub{(x_1\longc x_{k})\in\Ome_{k} \\ x\in D}} C(x-x_k) \\
     &= k \sum_{\sub{(x_1\longc x_{k})\in\Ome_{k} \\ x\in G}}
                                          (1-C(x))\, C(x-x_k) \\
     &= \tau |D|\cdot kT^{(k)}_D(D)
     - k\sum_{\sub{(x_1\longc x_{k})\in\Ome_{k} \\ x\in G}}
                                              C(x)\, C(x-x_k) \\
     &\ge \tau |D|\cdot kT^{(k)}_D(D)
                                  - k\tau^2|D|^2 T^{(k-1)}_D(D)
\end{align*}
(the inequality follows by allowing $x_k$ to take all values from $G$).

Finally, for any fixed $1\le i<j\le k$,
  $$ \sum_{x_i,x\in D} C(x-x_i)C(x-x_j)
       \le \sum_{x_i\in D} |(C+x_i)\cap(C+x_j)| \le |C|^2 = \tau^2|D|^2. $$
Consequently, $\sig_3\le\binom k2 T^{(k-1)}_D(D) \tau^2|D|^2$.
Therefore
\begin{align*}
   T_D^{(k+1)}(D)
       &\le T^{(k)}_D(D)|D| - \tau kT^{(k)}_D(D)|D|
          + k\tau^2 T^{(k-1)}_D(D)|D|^2 \\
       &\qquad\qquad + \binom k2\,\tau^2 T^{(k-1)}_D(D)|D|^2 \\
       &=(1-\tau k)T^{(k)}_D(D)|D|
                         + \frac12\,k(k+1) \tau^2 T^{(k-1)}_D(D)|D|^2.
\end{align*}
In view of the induction hypothesis~\refe{Sigma_large}, it suffices to prove
that
\begin{multline*}
  (1-\tau k)\Big(1-\frac14\,k(k-1)\tau\Big)
        + \frac12k(k+1)\tau^2\Big(1-\frac14\,(k-1)(k-2)\tau\Big) \\
        \le 1-\frac14\,k(k+1)\tau.
\end{multline*}
Expanding and simplifying, this inequality is equivalent to
  $$ (k^2-1)(k-2)\tau^2 -2(k^2 + k + 2)\tau + 4 \ge 0 $$
which clearly holds true under the assumptions $k\ge 2$ and
$\tau\le2/(k^2+k+2)$ made at the beginning of the proof.
\end{proof}

\begin{proof}[Proof of Theorem~\ref{t:arbG}]
Let $D:=A-A$, $\tau:=|G\stm D|/|D|$, and $k:=\lcl 1/\sqrt{\tau}\rcl+1$.
Notice that the assumption $|D|=(1-\eps)|G|$ implies $\tau=\eps/(1-\eps)$ and
$\eps=\tau/(\tau+1)$; along with $\eps\le2^{-5}$, this leads to
\begin{equation}\label{e:074}
  \frac1k > \frac{\sqrt\tau}{2\sqrt\tau+1}
  = \frac{\sqrt\eps}{2\sqrt\eps+\sqrt{1-\eps}} > 0.747\sqrt\eps.
\end{equation}

Furthermore, from $|D|<|A|^{1+\sqrt{\eps}/2}$ and $|A|\ge2^{10}$ we get
\begin{multline*}
   |A|^{k-1}|D|^k / |A|^{2k} \le|A|^{-(k+1)+k(1+\sqrt{\eps}/2)}
       = |A|^{k\sqrt{\eps}/2-1} \\
       < |A|^{\sqrt{\eps/\tau}/2+\sqrt{\eps}-1}
         = |A|^{\sqrt{1-\eps}/2+\sqrt{\eps}-1} < |A|^{-0.331} < 0.101.
\end{multline*}

As a result, applying the estimate \eqref{e:basic} and Lemma
\ref{l:Sigma_large},
\begin{align*}
  |A|^{2k}
     &\le T_D^{(k)}(D) \cdot (|A|^{k-1} + (\mu(A))^k) \\
     &\le \left( 1 - \frac14\,k(k-1)\tau \right)
                                   \cdot (|A|^{k-1} + (\mu(A))^k) |D|^k \\
     &< \frac34 \cdot \Big(0.101|A|^{2k} + (\mu(A))^k |D|^k\Big)
\end{align*}
whence
  $$ (\mu(A))^k |D|^k > 1.232\,|A|^{2k} $$
and then
  $$ \mu(A) > \frac{|A|^2}{|D|} \left(1.232\right)^{1/k}. $$
Finally, recalling~\refe{074},
  $$ 1.232^{1/k} > 1+\frac1k\ln(1.232) > 1 + \frac18\,\sqrt\eps. $$
\end{proof}


\section{Proof of Theorem~\reft{extD}}\label{s:extD}

\begin{proof}[Proof of Theorem~\reft{extD}]
Let $D:=A-A$, and let $\mu^{(k)}(A):=\max R_A^{(k)}(d_1\longc d_{k-1})$ where
the maximum extends over all $(k-1)$-tuples $(d_1\longc d_{k-1})\in D^{k-1}$
with pairwise distinct, nonzero components $d_i$.

Let $l:=\lfl1/(2\del(k-1))\rfl$; we notice that, in view of
  $$ \frac1{2\del(k-1)}-\frac1{3\del k} >
       \frac1{2\del(k-1)}-\frac1{3\del(k-1)}
       = \frac1{6\del(k-1)} > \frac1{6\del k}>1, $$
we have
\begin{equation}\label{e:l}
  \frac1{3\del k} < l \le \frac1{2\del(k-1)}.
\end{equation}

Similarly to~\refe{basic}, from~\refe{sumC} and~\refe{Ckl}, using H\"older's
inequality we get
\begin{align}
    |A|^{k(l+1)}
      &=  \left( \sum_{d_1,\dots, d_{l}\in D}
                      (R_A^{(l+1)}(d_1\longc d_l)) \right)^k \label{e:kl} \\
      &\le (T^{(l)}_D (D) )^{k-1} \cdot \sum_{d_1\longc d_{l}\in D}
                              (R_A^{(l+1)}(d_1\longc d_l))^k \notag \\
      &= (T^{(l)}_D(D))^{k-1} \cdot \sum_{d_1\longc d_{k-1}\in D}
                       (R_A^{(k)}(d_1\longc d_{k-1}))^{l+1}. \notag
\end{align}
Let $\sig_0$ denote the part of the sum in the right-hand side extending over
the $(k-1)$-tuples $(d_1\longc d_{k-1})$ with either $d_i=0$, or $d_i=d_j$
with some $i,j\in[1,n],\ i\ne j$, and let $\sig_1$ be the sum over the
$k$-tuples $(d_1\longc d_{k-1})$ with all components $d_i$ distinct from $0$
and from each other.

Using~\refe{l}, we obtain
\begin{align*}
  (T^{(l)}_D(D))^{k-1} \sig_0
      &\le |D|^{l(k-1)} \cdot k^2|A|^l\,\sum_{d_1\longc d_{k-2}\in D}
                                         R_A^{(k-1)}(d_1\longc d_{k-2}) \\
      &\le |A|^{(1+\del)l(k-1)} k^2 |A|^{k+l-1} \\
      &=   k^2 |A|^{k(l+1)-1+\del l(k-1)} \\
      &\le k^2 |A|^{k(l+1)-1/2} \\
      &\le \frac12 |A|^{k(l+1)}.
\end{align*}
Comparing this estimate with~\refe{kl}, we conclude that
$(T^{(l)}_D(D))^{k-1} \sig_1\ge\frac12 |A|^{k(l+1)}$. On the other hand, we
have $\sig_1\le(\mu^{(k)}(A))^l|A|^{k}$. Therefore,
  $$ (\mu^{(k)}(A))^l|A|^{k} (T^{(l)}_D(D))^{k-1}
                                         \ge \frac12\,|A|^{k(l+1)} $$
whence, by Corollary~\refc{intopt}
\begin{align*}
  \mu^{(k)}(A) &> ((3l2^{-l-1}|D|^{l}))^{-(k-1)/l}\cdot2^{-1/l}|A|^{k} \\
    &= \frac{2^{k-1}|A|}{|K|^{k-1}}\,((3/2)l)^{-(k-1)/l}2^{-1/l} \\
    &> \frac{2^{k-1}|A|}{|K|^{k-1}}\,(3l)^{-(k-1)/l}.
\end{align*}

Finally, using~\refe{l} again,
  $$ (3l)^{-(k-1)/l} > 1-3(k-1)\,\frac{\ln(3l)}{3l}
      > 1 - 3k\,\frac{\ln(1/k\del)}{1/k\del}
       > 1-3\del k^2\ln(1/k\del). $$
\end{proof}

\section{Proof of Theorem~\reft{fan}}\label{s:fan}

So far, we have dealt with finite sets, which can be identified with their
characteristic functions. More generally, instead of the sets, one can
consider \emph{any} finitely (or compactly) supported nonnegative function
$f$, and study the quantities
  $$ R^{(k)}_f(x_1\longc x_{k-1})
       := \sum_{x\in\supp f} f(x)f(x+x_1)\dotsb f(x+x_{k-1}),
                                          \quad x_1\longc x_{k-1}\in G $$
(cf.~\refe{Rdef}) aiming to show that some of these quantities are
``atypically large''.

In keeping with this interpretation, instead of \refe{sumC} and \refe{Ckl} we
use the identities
  $$ \sum_{x_1\longc x_{k-1}\in G}
                               R^{(k)}_f(x_1\longc x_{k-1}) = \|f\|_1^k $$
and
  $$ \sum_{x_1\longc x_{k-1}\in G}\big(R^{(k)}_f(x_1\longc x_{k-1})\big)^l
        = \sum_{y_1\longc y_{l-1}\in G}
                              \big(R^{(l)}_f(y_1\longc y_{l-1})\big)^k $$
implying
\begin{align}
  \|f\|_1^{2k+2}
   &\le \sum_{x_1,\dots,x_k} (R^{(k+1)}_f (x_1,\dots,x_k))^2
               \cdot |\supp (R^{(k+1)}_f)| \notag \\
   &\le T^{(k)}_D (D) \sum_{x\in G}\big(R_f^{(2)}(x) \big)^{k+1},
                                                         \label{f:basic_f}
\end{align}
where $D=\supp(f)$, cf.~\refe{basic}.

To illustrate this approach, consider a finite, nonempty subset $A$ of an
abelian group $G$, and let $D:=A-A$ and $f=r_A$; thus, $\supp(f)=D$, and
$\|f\|_1=|A|^2$. Suppose that $K:=|D-D|/|D|<|D|^\del$ for a sufficiently
small $\del>0$. If $G$ is cyclic and either infinite, or of prime order
$p\ge\frac32\,|D|-1$, then following the proofs of Theorems \reft{modp} and
\reft{intverD} one can show that there exists a nonzero group element $s$
such that
  $$ (f\circ f) (s) = | \{  (x,y,z,w) \in A^4\colon x+y-z-w = s\} |
           \ge (2+o_\del (1)) \frac{|A|^4}{K|D|}. $$

As another illustration, we prove Theorem \reft{fan}.

\begin{proof}[Proof of Theorem~\ref{t:fan}]
Our argument uses continuous analogs of the quantities and results developed
above in the discrete settings. We sketch the proof leaving it to the reader
to work out the missing details.

Let $D:=[-1,1]$ and
 $\ome:=\|f\|_1^{-2} \sup_{x\in \mathbb R\stm [-\del,\del]}(f\circ f)(x)$.
We have
  $$ \int_\R(f(x)-\|f\|_1)^2\,dx
                       = \|f\|_2^2 - \|f\|_1^2 = (\rho^2-1)\|f\|_1^2 $$
and by the Cauchy-Schwarz inequality, for any $s\in\R$,
\begin{align*} \ome
  \|f\|_1^2 &\ge (f\circ f)(s) \\
     &= \|f\|_1^2 + \int_\R f(x)(f(x+s)-\|f\|_1)\, dx \\
     &\ge\|f\|_1^2-\|f\|_2\, \sqrt{\rho^2-1}\,\|f\|_1.
\end{align*}
Consequently,
\begin{equation*}\label{e:ommm}
   \ome \ge 1-\rho\sqrt{\rho^2-1}.
\end{equation*}

Suppose first that $\rho\le e^{2\sqrt[4]{2\del}}$. Then
$\rho<1+4\sqrt[4]{2\del}$ and
$\rho\sqrt{\rho+1}<\rho\sqrt{2\rho}\le\sqrt{2}e^{3\sqrt[4]{2\del}}$, and it
follows that
  $$ \rho\sqrt{\rho^2-1}
        < 2\sqrt[8]{2\del} \cdot \sqrt{2}e^{3\sqrt[4]{2\del}}
           < 8\sqrt[8]{2\del}, $$
implying the assertion.

Suppose now that $\rho\ge e^{2\sqrt[4]{2\del}}$.

As a continuous version of \eqref{f:basic_f}, we have
\begin{align*}
  \|f\|_1^{2k+2}
     &\le \int_{-1}^1 (f\circ f)^{k+1}(x)\, dx \cdot T^{(k)}_D(D) \\
     &\le \left( 2\del \|f\|_2^{2k+2}
              + \ome^k \|f\|_1^{2k+2} \right) \cdot T^{(k)}_D(D) \\
     &=   \left( 2\del \rho^{2k+2}
            + \ome^k \right) \|f\|_1^{2k+2} \cdot T^{(k)}_D(D).
\end{align*}
On the other hand, using an argument similar to that in the proof of
Lemma~\refl{Int}, it is easy to see that $T^{(k)}_D(D)=k+1$. Therefore,
\begin{equation}\label{e:omegak}
  \ome^k \ge \frac1{k+1} - 2\del \rho^{2k+2}.
\end{equation}

To optimize, we let $L_1=\log_\rho(1/2\del)$, $L:=L_1-\log_\rho(L_1)$, and
$k:=\lfl L/2 \rfl-1$; thus, $L_1\ge 12$ by the assumptions. Furthermore, from
the inequality $\sqrt{z}\ln(1/z)\le 2\sqrt[4]{z}$ (valid for any positive
$z\le 1$), substituting $z=2\del$ we get
$\sqrt{2\del}\ln(1/2\del)\le\ln(\rho)$ whence $L_1\le\frac1{\sqrt{2\del}}$.
It follows that $\log_\rho(L_1)\le\frac12\log_\rho(1/2\del)=\frac12\,L_1$ and
then $L=L_1-\log_\rho(L_1)\ge \frac12\,L_1\ge 6$, implying $k\ge L/4$. Also,
  $$ (2k+2)\rho^{2k+2}\le L\rho^L\le
       L_1\rho^{L_1-\log_\rho(L_1)}=\rho^{L_1}=\frac1{2\del}. $$
 Therefore, from~\refe{omegak},
   $$ \ome^k\ge\frac1{2(k+1)}. $$
It follows that
  $$ \ome \ge e^{-\frac1k\ln(2(k+1))}
      \ge 1-\frac1k\ln(2(k+1)) \ge 1 - \frac4{L}\ln(L), $$
and to complete the proof we recall that $\frac12\,L_1\le L\le L_1$.
\end{proof}

\end{document}